\documentclass[12pt]{amsart}
\usepackage[T1]{fontenc}
\usepackage{amssymb}
\let\SavedRightarrow=\Rightarrow
\usepackage{marvosym}
\let\Rightarrow=\SavedRightarrow

\newcommand{\Aaa }{\mathcal A}
\newcommand{\Raa }{\mathcal R}
\newcommand{\Bee }{\mathcal B}
\newcommand{\Cee }{\mathcal C}

\newcommand{\Qee }{\mathcal Q}
\newcommand{\Pee }{\mathcal P}
\newcommand{\Fee }{\mathcal F}
\newcommand{\Uee }{\mathcal U}
\newcommand{\Tee }{\mathcal T}

\newcommand{\See }{\mathcal S}

\newcommand{\cl}{\operatorname{cl}}
\renewcommand{\int}{\operatorname{Int}}

\newtheorem{thm}{Theorem}[section]
\newtheorem{pro}[thm]{Proposition}
\newtheorem{lem}[thm]{Lemma}
\newtheorem{cor}[thm]{Corollary}
\newtheorem{rem}[thm]{Remark}
\newtheorem*{MainTheorem1}{Theorem \ref{represent}}
\newtheorem*{MainTheorem2}{Theorem \ref{4.1}}
\begin{document}

\title{Very I-favorable  spaces}

\author{A. Kucharski}
\address{Institute of Mathematics, University of Silesia, ul. Bankowa 14, 40-007 Katowice}
\email{akuchar@ux2.math.us.edu.pl}

\author{Sz. Plewik}
\address{Institute of Mathematics, University of Silesia, ul. Bankowa 14, 40-007 Katowice}
\email{plewik@math.us.edu.pl}

\author{V. Valov}
\address{Department of Computer Science and Mathematics, Nipissing University,
100 College Drive, P.O. Box 5002, North Bay, ON, P1B 8L7, Canada}
\email{veskov@nipissingu.ca}
\thanks{The third author was supported by NSERC Grant 261914-08}

\begin{abstract}
We prove that a Hausdorff space $X$ is very $\mathrm I$-favorable if
and only if $X$ is the almost limit space of a $\sigma$-complete inverse
system consisting of (not necessarily Hausdorff) second countable
spaces and surjective d-open bonding maps. It is also shown that the
class of Tychonoff very $\mathrm I$-favorable spaces with respect to
the co-zero sets coincides with the d-openly generated spaces.
\end{abstract}

\maketitle \markboth{}{Very I-favorable}

\section{Introduction}
The classes of I-favorable and very I-favorable spaces were
introduced by P. Daniels, K. Kunen and H. Zhou \cite{dkz}. Let us
recall the corresponding definitions. Two players are playing the so
called \textit{open-open game} in a space $(X,\Tee_X)$, a round
consists of player I choosing a nonempty open set $U\subset X$ and
player II a nonempty open set $V\subset U$; I wins if the union of
II's open sets is dense in $X$, otherwise II wins. A space $X$ is
called I-\textit{favorable} if player I has a winning strategy. This
means that there exists a function $\sigma:\bigcup\{\Tee_X^n:n\geq
0\}\to\Tee_X$ such that for each game
$$\sigma(\varnothing),B_0,\sigma(B_0),B_1,\sigma(B_0,B_1),B_2,\ldots,
B_n,\sigma(B_0,\ldots,B_n),B_{n+1},\ldots$$
the union  $\bigcup_{n\geq 0}B_n$ is dense in $X$, where
$\emptyset\not=\sigma(\varnothing)\in \Tee_X$ and
$B_{k+1}\subset\sigma(B_0,B_1,..,B_k)\not=\emptyset$ and
$\emptyset\not=B_k \in \Tee_X$ for $k\geq 0$.

A family $\Cee\subset\mathcal [\Tee_X]^{\leq \omega}$ is said to be
a \textit{club} if:  (i) $\Cee$ is closed under increasing
$\omega$-chains, i.e., if $C_1\subset C_2\subset...$ is an
increasing $\omega$-chain from $\Cee$, then $\bigcup_{n\geq
1}C_n\in\Cee$; (ii) for any $B\in[\Tee_X]^{\leq\omega}$ there exists
$C\in\Cee$ with $B\subset C$.

Let us recall \cite[p. 218]{kun}, that $C\subset_c\Tee_X$  means that
for any nonempty $V\in\Tee_X$ there
exists $W\in C$ such that if  $U\in C$ and $U\subset W$, then $U\cap
V\neq\varnothing$.
A space $X$ is  \textit{$\mathrm
I$-favorable} if and only if the family $$\{\Pee\in [\Tee_X]^{\leq
\omega}:\Pee\subset_c\Tee_X\}$$ contains a club, see \cite[Theorem 1.6]{dkz}.

A space $X$ is called \textit{very $\mathrm
I$-favorable} if the family $$\{\Pee\in [\Tee_X]^{\leq
\omega}:\Pee\subset_!\Tee_X\}$$ contains a club. Here,
$\Pee\subset_!\Tee_X$ means that for any $\See\subset \Pee$ and $x
\notin \cl_X \bigcup\See$, there exists $W\in\Pee$ such that $x\in W$
and $W\cap\bigcup \See=\emptyset$. It is easily seen that
$\Pee\subset_!\Tee_X$ implies $\Pee\subset_c\Tee_X$.

 It was shown by the first two authors in \cite{kp8} that a compact Hausdorff
space is I-favorable if and only if it can be represented as the
limit of a $\sigma$-complete (in the sense of Shchepin \cite{s76})
inverse system consisting of I-favorable
compact metrizable  spaces and skeletal bonding maps, see also \cite{kp7}
and \cite{kp9}. For similar characterization of I-favorable spaces with respect to
co-zero sets, see \cite{vv}.   Recall that a continuous map $f:X \to Y$ is called
\textit{skeletal} if the set $\int_Y\cl_Y f(U)$ is non-empty, for
any $U\in \Tee_X$, see \cite{mr}.

In this paper we show that there exists an analogy between the
relations I-favorable spaces - skeletal maps and very I-favorable
spaces - d-open maps (see Section 2 for the definition of d-open
maps). The following two theorems are our main results:

\begin{MainTheorem1}
A regular space $X$ is very $\mathrm I$-favorable if and only if
$X=\displaystyle\mathrm{a}-\underleftarrow{\lim}S$, where $S=\{X_A,
q^A_B,\Cee\}$  is a $\sigma$-complete inverse system such that all
$X_A$ are $($not-necessarily Hausdorff$)$ spaces with countable
weight and the bonding maps $q^A_B$ are d-open and onto.
\end{MainTheorem1}

\begin{MainTheorem2}
A completely regular space $X$ is very $\mathrm I$-favorable with
respect to the co-zero sets if and only if $X$ is $\mathrm d$-openly
generated.
\end{MainTheorem2}

Here, a completely regular space $X$ is d-\textit{openly generated}
if there exists a $\sigma$-complete inverse system $\displaystyle
S=\{X_\sigma, \pi^{\sigma}_\varrho, \Gamma\}$ consisting of
separable metric spaces $X_\sigma$ and d-open surjective bonding
maps $\pi^{\sigma}_\varrho$ with $X$ being embedded in
$\displaystyle\underleftarrow{\lim}S$ such that
$\pi_\sigma(X)=X_\sigma$ for each $\sigma\in\Gamma$.

Theorem 4.1 implies the following characterization of
$\kappa$-metrizable compacta (see Corollary 4.3), which provides an
answer of a question from \cite{vv}: A compact Hausdorff space is
very $\mathrm I$-favorable with respect to the co-zero sets if and
only if $X$ is $\kappa$-metrizable.

\section{Very I-favorable spaces and d-open maps}

T. Byczkowski and R. Pol \cite{byp} introduced nearly open sets and
nearly open maps as follows. A subset of a topological space is
\textit{nearly open} if it is in the interior of its closure. A map
is  \textit{nearly open} if the image of every open subset is nearly
open. Continuous nearly open maps were called d-\textit{open} by M.
Tkachenko \cite{tk}. Obviously, every d-open map is skeletal.

\begin{pro}\label{d-open}
Let $(X,\Tee_X)$ and $(Y,\Tee_Y)$ be topological spaces and $f:X \to Y$
a continuous function. Then the following conditions are equivalent:
\begin{enumerate}
\item $f$ is d-open;
\item $\cl_Xf^{-1}(V)=f^{-1}(\cl_Y V)$ for any open $V\subset Y$;
\item $f(U)\subset \int_Y\cl_Y f(U)$ for every open subset $U\subset X$;
\item $\{f^{-1}(V):V\in\Tee_Y\}\subset_!\Tee_X$.
\end{enumerate}
\end{pro}

\begin{proof}
The implication $(1)\Rightarrow (2)$ was established in \cite[Lemma
5]{tk}. Obviously $(3)\Rightarrow (1)$. Let us prove the implication
$(2)\Rightarrow (3)$. Suppose
$U\subset X$ is open. Then we have
$X\setminus f^{-1} (\int_Y \cl_Y f(U) ) \subset X \setminus U$.
Indeed, $Y\setminus\int_Y\cl_Y f(U)=\cl_Y(Y\setminus\cl_Y
f(U))$ and by $(2)$ we get $$f^{-1}(\cl_Y(Y\setminus\cl_Y
f(U)))=\cl_X(f^{-1}(Y \setminus\cl_Y f(U))).$$ But
$\cl_X(f^{-1}(Y \setminus\cl_Y f(U)))= \cl_X(X\setminus f^{-1}(\cl_Y f(U)))$
and
$$ X\setminus f^{-1}(\cl_Y f(U))\subset X\setminus\cl_X f^{-1}(
f(U))\subset X\setminus \cl_X U\subset X\setminus U.
$$  Hence $f(U)\cap Y\setminus\int_Y\cl_Y f(U) =\emptyset$ and
$f(U)\subset\int_Y\cl_Y f(U)$.

To show $(4)\Rightarrow (2)$, assume that
$\{f^{-1}(V): V \in \Tee_Y\} \subset_!\Tee_X$.  Since $f$ is continuous
we get $\cl_X f^{-1}(V)\subset f^{-1}(\cl_Y V)$ for   any open set
$V \subset Y$.  We shall show that $f^{-1}(\cl_Y V) \subset \cl_X f^{-1}(V)
 $ for   any open  $V \subset Y$.    Suppose there
exists an open set $V\subset Y$ such that $$ f^{-1}(\cl_Y
V)\setminus\cl_X f^{-1}(V)\ne\emptyset.$$ Let $x\in f^{-1}(\cl_Y
V)\setminus\cl_X f^{-1}(V)$ and $\See=\{f^{-1}(V)\}$. Since
$x\not\in\cl_X\bigcup\See=\cl_X f^{-1}(V)$, there is an open set
$U\in\Bee_Y$ such that $x\in f^{-1}(U)$ and $f^{-1}(U)\cap
f^{-1}(V)=\emptyset.$ Therefore, $f(x)\in U\cap\cl_Y V$ which
contradicts $V\cap U=\emptyset$.

Finally, we can show that (2) yields $\{f^{-1}(V):V\in\Tee_Y\}
\subset_!\Tee_X$. Indeed, let $\See\subset\{f^{-1}(V): V\in
\Tee_Y\}$ and $x\not\in\cl_X\bigcup\See$. Then there is $U\in\Tee_Y$
such that $\bigcup \See=f^{-1}(U)$. Hence,
$\cl_X\bigcup\See=f^{-1}(\cl_Y U)$. Put $$W=f^{-1}(Y\setminus \cl_Y
U).$$
 We have $x\in W$ and $W\cap\cl_X\bigcup\See=\emptyset$.
\end{proof}
\begin{rem}\label{2.2}
If, under the hypotheses of Proposition \ref{d-open}, there exists a base
$\Bee_Y\subset\Tee_Y$ with
$\{f^{-1}(V):V\in\Bee_Y\}\subset_!\Tee_X$, then $f$ is d-open.
\end{rem}

Indeed, we can follow the proof of the implication $(4)\Rightarrow
(2)$ from Proposition \ref{d-open}. The only difference is the choice of the
family $\See$. If there exists $x\in f^{-1}(\cl_Y V)\setminus\cl_X
f^{-1}(V)$ for some open $V\subset Y$, we choose
$\See=\{f^{-1}(W):W\in\Bee_Y{~}\mbox{ and }{~}W\subset V\}$.

Next lemma was established in \cite[Lemma 9]{tk}.

\begin{lem}\label{d-open1}
Let $f:X\to Y$  and $g\colon Y\to Z$ be continuous maps with $f$
being surjective. Then $g$ is d-open provided so is $g\circ f$. \hfill $\Box$
\end{lem}

Let  $X$ be a topological space equipped with a topology $\Tee_X$
and $\Qee\subset\Tee_X$. Suppose that  there exists a function
$\sigma :\bigcup\{ \Qee^n: n\geq 0\} \to \Qee $  such that if  $B_0,
B_1, \ldots\;$ is a sequence of non-empty elements of $\Qee$ with
$B_0 \subset \sigma(\varnothing)$ and $B_{n+1} \subset\sigma((B_0,
B_1, \ldots, B_n))$ for all $n\in \omega$, then $\{B_n:n\in
\omega\}\cup\{\sigma((B_0, B_1, \ldots, B_n)):n\in
\omega\}\subset_!\Qee$. The function $\sigma$ is called a
\textit{strong winning strategy in $\Qee$}. If $\Qee=\Tee_X$,
$\sigma$ is called a strong winning strategy. It is clear that if
$\sigma$ is strong winning strategy, then it is a winning strategy
for player I in the open-open game.

\begin{lem}\label{2.4}
Let $\sigma :\bigcup\{ \Qee^n: n\geq 0\} \to \Qee $ be a strong
winning strategy in $\Qee$, where $\Qee$ is a family of  open subsets of $X$.
 Then $\Pee\subset_!\Qee$ for every family
$\Pee\subset\Qee$ such that $\Pee$ is closed under $\sigma$ and
finite intersections.
\end{lem}
\begin{proof} Let $\Pee\subset\Qee$ be closed under $\sigma$ and finite
intersections.
 Fix a family
$S\subset \Pee$ and $x\not\in\cl\bigcup S$.
If $\sigma(\varnothing)\cap\bigcup S\not =\emptyset$, then take an element
$U\in S$ such that $\sigma(\varnothing)\cap U\not =\emptyset$ and put
$V_{0}= \sigma(\varnothing)\cap U\in \Pee.$
If $\sigma(\varnothing)\cap\bigcup S=\emptyset$, then put
$V_{0}=\sigma(\varnothing)\in \Pee.$ Assume that sets
$V_0,\dots ,V_n \in \Pee $ are just defined.
If $\sigma(V_0,\dots ,V_n)\cap\bigcup S\not =\emptyset$,
then take an element $U\in S$ such that
$\sigma(V_0,\dots ,V_n)\cap U\not =\emptyset$ and put
$V_{n+1}= \sigma(V_0,\dots ,V_n)\cap U\in \Pee.$
If $\sigma(V_0,\dots ,V_n)\cap\bigcup S=\emptyset$, then put
$V_{n+1}=\sigma(V_0,\dots ,V_n)\in \Pee.$ Take a subfamily
$$\Uee=\{V_{k}: V_k\cap \bigcup S\not =\emptyset \mbox{ and }
k\in \omega \}\subset\Qee.$$ Since $\sigma $ is strong strategy,
then   $\bigcup\{V_n:n\in \omega\}$ is dense in $X$. Hence
$\cl\bigcup \Uee=\cl\bigcup S$. Since $\{V_n:n\in \omega\}\cup
\{\sigma((V_0, V_1, \ldots, V_n)):n\in \omega\}\subset_!\Qee$
there exists $V\in \{V_n:n\in \omega\}\cup
\{\sigma((V_0, V_1, \ldots, V_n)):n\in \omega\}\subset \Pee$ such that
$x\in V$ and $V\cap \bigcup S=\emptyset$.
\end{proof}

\begin{pro}\label{club-strategy}
Let $X$ be a topological space and $\Qee\subset\Tee_X$ be a family
closed under finite intersection. Then there is a strong winning
strategy  $\sigma :\bigcup \{ \Qee^n: n\geq 0\}\to \Qee $ in $\Qee$
if and only if the family $\{\Pee\in [\Qee]^{\leq \omega}:
\Pee\subset_! \Qee\}$ contains a club $\Cee$ such that every
$A\in\Cee$ is closed under finite intersections.
\end{pro}

\begin{proof}
If there is a club $\Cee\subset\{\Pee\in [\Qee]^{\leq \omega}:
\Pee\subset_! \Qee\}$, then following the arguments from
\cite[Theorem 1.6]{dkz} one can construct a strong winning strategy
in $\Qee$.

Suppose there exists a strong winning strategy $\sigma :\bigcup \{
\Qee^n: n\geq 0\} \to \Qee $. Let $\Cee$ be the family of all
countable subfamilies $A\subset\Qee$ such that $A$ is closed under
$\sigma$ and finite intersections. The family
$\Cee\subset [\Qee]^{\leq \omega}$ is a club.
Obviously, $\Cee$ is closed under increasing $\omega$-chains. If
$B\in [\Qee]^{\leq \omega}$, there exists a countable family
$A_B\subset\Qee$ which contains $B$ and is closed under $\sigma$ and
finite intersections. So, $A_B\in\Cee$. According to Lemma \ref{2.4},
$A\subset_! \Qee$ for all
$A\in\Cee$.
\end{proof}

\begin{cor}\label{2.6}
A Hausdorff space $(X,\Tee)$ is very $\mathrm I$-favorable if and only
if the family $\{\Pee\in [\Tee]^{\leq \omega}: \Pee\subset_! \Tee\}$
contains a club $\Cee$ with the following properties:
\begin{itemize}
\item[(i)] every $A\in\Cee$ covers $X$ and it is closed under finite
intersections;
\item[(ii)] for any two different points $x,y\in  X$ there exists
$A\in\Cee$ containing two disjoint elements $U_x,U_y\in A$ with
$x\in U_x$ and $y\in U_y$;
\item[(iii)] $\bigcup\Cee=\Tee$. \hfill $\Box$
\end{itemize}
\end{cor}

The next proposition shows that every space $X$ having a base
$\Bee_X$ such that the family $\{\Pee\in [\Bee_X]^{\leq \omega}:
\Pee\subset_! \Bee_X\}$ contains a club is very I-favorable.

\begin{pro}\label{2.7} If there exists  a base $\Bee$ of  $X$ such that
the family $\{\Pee\in [\Bee]^{\leq \omega}: \Pee\subset_! \Bee\}$
contains a club,  then the family $\{\Pee\in [\Tee_X]^{\leq \omega}:
\Pee\subset_! \Tee_X\}$  contains a club too.
\end{pro}

\begin{proof}
If there exists  a base $\Bee$ of  $X$ such that  the family
$\{\Pee\in [\Bee]^{\leq \omega}: \Pee\subset_! \Bee\}$ contains a
club,  then there exists a strong winning strategy in $\Bee$.
Therefore, player I has winning strategy in the open-open game
$G(\Bee)$ (i.e., the open-open game when each player chooses a set
from $\Bee$). This implies that $X$ satisfies the countable chain
condition, otherwise the strategy for player II to choose at each
stage a nonempty subset of a member of a fixed uncountable maximal
disjoint collection of elements of $\Bee$ is winning (see
\cite[Theorem 1.1(ii)]{dkz} for a similar situation). Consequently,
every nonempty open subset $G\subset X$ contains a countable
disjoint open family whose union is dense in $G$ (just take a
maximal disjoint open family in $G$). Now, for each element
$U\in\Tee_X\setminus \Bee$ we assign a countable family
$\Aaa_U\subset\Bee$ of pairwise disjoint open subsets of $U$ such
that $\cl\bigcup\Aaa_U=\cl U$. If $U\in \Bee$, then we assign
$\Aaa_U=\{U\}$. Let $\Cee\subset\{\Pee\in [\Bee]^{\leq \omega}:
\Pee\subset_! \Bee\}$ be a club. Put
$$\Cee'=\{ A \cup \Qee: \Qee\in\Cee{~}\mbox{and}{~}
A\in[\Tee_X]^{\leq\omega}{~}\mbox{with}
{~}\Aaa_U\subset\Qee{~}\mbox{for all}{~}U\in A\}.$$ \indent
First, observe that if $A\cup\Qee_A\subset
D\cup\Qee_D$ and $A\cup\Qee_A, D\cup\Qee_D\in\Cee'$, then
$\Qee_A\subset\Qee_D$. Indeed, if $U\in \Qee_A\subset \Bee$ then $U\in
D\cup\Qee_D$ and $U\in \Bee$. If $U\in D$, then  we
get $\{U\}=\Aaa_U\subset \Qee_D$ (i.e. $U\in \Qee_D$).
Therefore, if we have a chain
$\{A_n\cup\Qee_{A_n}:n\in\omega\}\subset\Cee'$, then
\[ \bigcup \{A_n\cup\Qee_{A_n}:n\in\omega\}=\bigcup_{n\in\omega}
A_n\cup\bigcup_{n\in\omega}\Qee_{A_n}\in\Cee'\]

The absorbing property (i.e. for every $A\in[\Tee_X]^{\leq \omega}$
there is an element $\Pee\in\Cee'$ such that $A\subset \Pee$) for
$\Cee'$ is obvious. So,
$\Cee'\subset [\Tee_X]^{\leq \omega}$ is a club.

It remains to prove that
$A\cup\Qee\subset_!\Tee_X$ for every $A\cup\Qee\in\Cee'$. Fix a
subfamily $\See\subset A\cup\Qee$ and $x\not\in\cl\bigcup\See$.
Define $$\See'=\{U\in\See:U\in\Qee\}\cup\bigcup\{\Aaa_U:U\in A\}$$ and
note that $\cl\bigcup\See=\cl\bigcup\See'$. The last equality
follows from the inclusion $\bigcup\See'\subset \bigcup\See$ and the
fact that $\bigcup\Aaa_U$ is dense in $U$ for every $U\in A$. So, if
$x\not\in\cl\bigcup\See$ then $x\not\in\cl\bigcup\See'$. Since
$\See'\subset\Qee\in\Cee$ there is $G\in\Qee$ such that $x\in G$ and
$G\cap \cl\bigcup \See'=\emptyset$
\end{proof}
If $X$ is a completely regular space, then $\Sigma_X$ denotes the collection
of all co-zero sets in $X$.
\begin{cor}\label{2.8}
Let $X$ be a completely regular space and $\Bee\subset \Sigma_X $
a base for  $X$. If $\{\Pee\in
[\Bee]^{\leq \omega}: \Pee\subset_! \Bee\}$ contains a club,  then
the family $\{\Pee\in [\Sigma_X]^{\leq \omega}: \Pee\subset_!
\Sigma_X\}$ contains  a club too.
\end{cor}

\begin{proof}
The proof of previous proposition works in the present situation.
The only modification is that for each $U\in\Sigma_X\setminus \Bee$
we assign a countable family $\Aaa_U\subset\Bee$ of pairwise
disjoint co-zero subsets of $U$ such that $\cl\bigcup\Aaa_U=\cl U$.
Such $\Aaa_U$ exists. For example, any maximal disjoint family of
elements from $\Bee$ which are contained in $U$ can serve as
$\Aaa_U$. The new club is the family $$\Cee'=\{A\cup\Qee :
\Qee\in\Cee{~}\mbox{and}{~}
A\in[\Sigma_X]^{\leq\omega}{~}\mbox{with}
{~}\Aaa_U\subset\Qee{~}\mbox{for all}{~}U\in A\},$$
where $\Cee \subset \{\Pee\in
[\Bee]^{\leq \omega}: \Pee\subset_! \Bee\}$ is a club.

\end{proof}

\section{Inverse systems with d-open bounding maps}

Recall some facts from \cite{kp8}. Let $\Pee$ be an open family in a
topological space $X$ and $x,y\in X$. We say that $x\sim_{\Pee} y$
if and only if $x\in V\Leftrightarrow y\in V$ for every $V\in\Pee.$
The family of all sets $[x]_{\Pee}=\{y:y\sim_{\Pee}x\}$ is denoted
by $X/\Pee$. There exists a mapping $q:X\to X/\Pee$ defined by
$q[x]=[x]_{\Pee}$. The set $X/\Pee$ is equipped with the topology
$\Tee_{\Pee}$ generated by all images $q(V)$, $V\in\Pee$.

\begin{lem}\cite[Lemma 1]{kp8}\label{quotion}
The mapping $q:X\to X/\Pee$ is continuous provided $\Pee$ is an open
family $X$ which is closed under finite intersection. Moreover, if
$X=\bigcup \Pee$, then the family $\{q(V): V\in \Pee\}$ is a base
for the topology $\Tee_{\Pee}$. \hfill $\Box$
\end{lem}

\begin{lem}\label{d-open2}
Let a space $X$ be the limit of a inverse system $\{ X_\sigma,
\pi^\sigma_\varrho,\Sigma\}$ with surjective bonding maps
$\pi^\sigma_\varrho$. Then $\pi^\sigma_\varrho$ are d-open if and
only if each projection $p_\sigma:X\to X_\sigma$ is d-open.
\end{lem}

\begin{proof}
Assume all $\pi^\sigma_\varrho$ are d-open. It suffices to show that
$p_\rho\big((p_\sigma)^{-1}(U)\big)$ is dense in some open subset of
$X_\rho$ for any open $U\subset X_\sigma$, where $\sigma\geq \rho$.
Since $p^\sigma_\rho$ is d-open and
$p_\rho\big((p_\sigma)^{-1}(U)\big)=p^\sigma_\rho(U)$, the proof is
completed. Conversely, if the limit projections are d-open, then, by
Lemma \ref{d-open1}, the bonding maps are also d-open.
\end{proof}

We say that a space $X$ is \textit{an almost limit} of the inverse
system $\displaystyle S=\{X_\sigma, \pi^{\sigma}_\varrho, \Gamma\}$,
if $X$ can be embedded in $\displaystyle\underleftarrow{\lim}S$ such
that $\pi_\sigma(X)=X_\sigma$ for each $\sigma\in\Gamma$. We denote
this by $X=\displaystyle\mathrm{a}-\underleftarrow{\lim}S$, and it
implies that $X$ is a dense subset of
$\displaystyle\underleftarrow{\lim} S$.

\begin{thm}\label{represent}
A Hausdorff space $X$ is very $\mathrm I$-favorable if and only if
$X=\displaystyle\mathrm{a}-\underleftarrow{\lim}S$, where $S=\{X_A,
q^A_B,\Cee\}$  is a $\sigma$-complete inverse system such that all
$X_A$ are $($not-necessarily Hausdorff$)$ spaces with countable
weight and the bonding maps $q^A_B$ are d-open and onto.
\end{thm}

\begin{proof}
Suppose $(X,\Tee)$ is very I-favorable. By Corollary \ref{2.6},
there exist a club $\Cee\subset\{\Pee\in [\Tee]^{\leq \omega}:
\Pee\subset_! \Tee\}$ satisfying conditions (i)-(iii).  For every
$A\in\Cee$ consider the space $X_A=X/A$ and the map $q_A\colon X\to
X_A$. Since each $A$ is a cover of $X$ closed under finite
intersections, by Lemma \ref{quotion}, $q_A$ is a continuous
surjection and $\{q_A(U):U\in A\}$ is a countable base for $X_A$.
Moreover, $q_A^{-1}(q_A(U))=U$ for all $U\in A$, see \cite{kp8}.
This, according to Remark \ref{2.2}, implies that each $q_A$ is
d-open (recall that $A\subset_!\Tee)$). If $A,B\in\Cee$ with
$B\subset A$, then there exists a map $q^A_B\colon X_A\to X_B$ which
is continuous because $(q^A_B)^{-1}(q_B(U))=q_A(U)$ for every $U\in
B$. The maps $q^A_B$ are also d-open, see Lemma \ref{d-open2}. In
this way we obtained the inverse system $S=\{X_A, q^A_B,\Cee\}$
consisting of spaces with countable weight and d-open bonding maps.
Since $\Cee$ is closed under increasing chains, $S$ is
$\sigma$-complete. It remains to show that the map $h\colon
X\to\underleftarrow{\lim}S$, $h(x)=(q_A(x))_{A\in\Cee}$, is a dense
embedding. Let $\pi_A\colon\underleftarrow{\lim}S\to X_A$,
$A\in\Cee$, be the limit projections of $S$. The  family
$\{\pi_A^{-1}(q_A(U)):U\in A, A\in\Cee\}$ is a base for the topology
of $\underleftarrow{\lim}S$. Since
$h^{-1}\big(\pi_A^{-1}(q_A(U))\big)=U$ for any $U\in A\in\Cee$, $h$
is continuous and $h(X)$ is dense in $\underleftarrow{\lim}S$.
Because $\Cee$ satisfies condition (ii) (see Corollary \ref{2.6}),
$h$ is one-to-one. Finally, since $h(U)=h(X)\cap \pi_A^{-1}(q_A(U))$
for any $U\in A\in\Cee$ (see \cite[the proof of Theorem 11]{kp8} and
$\Cee$ contains a base for $\Tee$, $h$ is an embedding.

Suppose now that $X=\displaystyle\mathrm{a}-\underleftarrow{\lim}S$,
where $S=\{X_A, q^A_B,\Cee\}$  is a $\sigma$-complete inverse system
such that all $X_A$ are spaces with countable weight and the bonding
maps $q^A_B$ are d-open and onto. Then, by Lemma \ref{d-open2}, all
limit projections $\pi_A\colon\underleftarrow{\lim}S\to X_A$,
$A\in\Cee$, are d-open. Since $X$ is dense in
$\underleftarrow{\lim}S$, any restriction $q_A=\pi_A|X\colon X\to
X_A$ is also d-open. Moreover, all $q_A$ are surjective (see the
definition of $\displaystyle\mathrm{a}-\underleftarrow{\lim}$).
Then, according to Proposition \ref{d-open},
$\{q_A^{-1}(U):U\in\Tee_A\}\subset_!\Tee$, where $\Tee_A$ is the
topology of $X_A$. Consequently, if $\Bee_A$ is a countable base for
$\Tee_A$, we have $\Pee_A=\{q_A^{-1}(U):U\in\Bee_A\}\subset_!\Tee$.
The last relation implies $\Pee_A\subset_!\Bee$ with
$\Bee=\bigcup\{\Pee_A:A\in\Cee\}$ being a base for $\Tee$. Let us
show that $\Pee=\{\Pee_A:A\in\Cee\}$ is a club in $\{\Qee\in
[\Bee]^{\leq\omega}:\Qee\subset_!\Bee\}$. Since $S$ is
$\sigma$-complete, the supremum of any increasing sequence from
$\Cee$ is again in $\Cee$. This implies that $\Pee$ is closed under
increasing chains.  So, it remains to prove that for every countable
family $\{U_j:j=1,2,..\}\subset\Bee$ there exists $A\in\Cee$ with
$U_j\in\Pee_A$ for all $j\geq 1$. Because every $U_j$ is of the form
$q_{A_j}^{-1}(V_j)$ for some $A_j\in\Cee$ and $V_j\in\Bee_{A_j}$,
there exists $A\in\Cee$ with $A>A_j$ for each $j$. It is easily seen
that $\Pee_A$ contains the family $\{U_j:j\geq 1\}$ for any such
$A$. Therefore, $\Pee$ is a club in $\{\Qee\in
[\Bee]^{\leq\omega}:\Qee\subset_!\Bee\}$. Finally, according to
Proposition \ref{2.7}, the family $\{\Qee\in
[\Tee]^{\leq\omega}:\Qee\subset_!\Tee\}$ also contains a club.
Hence, $X$ is very I-favorable.
\end{proof}

It follows from Theorem \ref{represent} that every dense subset of a
space from each of the following classes is very I-favorable:
products of first
 countable spaces, $\kappa$-metrizable compacta. More generally, by
\cite[Theorem 2.1 (iv)]{va1}, every space with a lattice of d-open
maps is very I-favorable.

The next theorem provides another examples of very I-favorable
spaces.

\begin{thm}\label{k-adic}
Let $f:X\xrightarrow{onto}Y$ be a perfect map with $X,Y$ being
regular spaces. Then $Y$ is very $\mathrm I$-favorable, provided so
is $X$.
\end{thm}

\begin{proof}
This theorem was established in \cite{dkz} when $X$ and $Y$ are
compact. The same proof works in our more general situation.
\end{proof}

\begin{cor}\label{k-dic}
Every continuous image under a perfect map of a space possessing a
lattice of d-open maps is very I-favorable. \hfill $\Box$
\end{cor}

\section{very I-favorable spaces with respect to the co-zero sets}

We say that a space $X$ is \textit{very I-favorable with respect to
the co-zero sets} if there exists a strong winning strategy $\sigma
:\bigcup\{ \Sigma_X^n: n\geq 0\} \to\Sigma_X$, where $\Sigma_X$
denotes the collection of all co-zero sets in $X$. By Proposition
\ref{club-strategy}, this is equivalent to the existence of a club in the family
$\{\Pee\in [\Sigma_X]^{\leq \omega}: \Pee\subset_! \Sigma_X\}$.

A completely regular space $X$ is d-\textit{openly generated} if $X$
is the almost limit of a $\sigma$-complete inverse system
$\displaystyle S=\{X_\sigma, \pi^{\sigma}_\varrho, \Gamma\}$
consisting of separable metric spaces $X_\sigma$ and d-open
surjective bonding maps $\pi^{\sigma}_\varrho$.

\begin{thm}\label{4.1}
A completely regular space $X$ is very $\mathrm I$-favorable with
respect to the co-zero sets if and only if $X$ is $\mathrm d$-openly
generated.
\end{thm}

\begin{proof}
Suppose $X$ is very $\mathrm I$-favorable with respect to the
co-zero sets and $\sigma :\bigcup\{ \Sigma_X^n: n\geq 0\}
\to\Sigma_X$ is a strong winning strategy in $\Sigma_X$. We place
$X$ as a $C^*$-embedded subset of a Tychonoff cube $\mathbb I^A$. If
$B\subset A$, let $\pi_B\colon\mathbb I^A\to\mathbb I^B$ be the
natural projection and $p_B$ be restriction map $\pi_B|X$. Let also
$X_B=p_B(X)$. If $U\subset X$ we write $B\in k(U)$ to denote that
$p_{B}^{-1}\big(p_{B}(U)\big)=U$.

\textit{Claim $1.$ For every $U\in\Sigma_X$ there exists a
countable $B_U\subset A$ such that $B_U\in k(U)$ with $p_{B_U}(U)$
being a co-zero set in $X_{B_U}$}.

For every $U\in\Sigma_X$ there exists a continuous function
$f_U\colon X\to [0,1]$ with $f_U^{-1}\big((0,1]\big)=U$. Next,
extend $f_U$ to a continuous function $g\colon\mathbb I^A\to [0,1]$
(recall that $X$ is $C^*$-embedded in $\mathbb I^A$). Then, there
exists a countable set $B_U\subset A$ and a function $h\colon\mathbb
I^{B_U}\to [0,1]$ with $g=h\circ\pi_{B_U}$. Obviously,
$U=p_{B_U}^{-1}\big(h^{-1}\big((0,1]\big)\cap p_{B_U}(X)\big)$,
which completes the proof of the claim.

Let $\Bee=\{U_\alpha:\alpha<\tau\}$ be a base for the topology of
$X$ consisting of co-zero sets such that for each $\alpha$ there
exists a finite set $H_\alpha\subset A$ with $H_\alpha\in
k(U_\alpha)$. For any finite set $C\subset A$ let $\gamma_C$ be a
fixed countable base for $X_C$.

\textit{Claim $2$. For every countable $B\subset A$ there exists a
countable set $\Gamma\subset A$ containing $B$ and a countable
family $\Uee_\Gamma\subset\Sigma_X$ satisfying the following
conditions:
\begin{itemize}
\item[(i)] $\Uee_\Gamma$ is closed under $\sigma$ and finite intersections;
\item[(ii)] $\Gamma\in k(U)$ for all $U\in\Uee_\Gamma$;
\item[(iii)] $\Bee_\Gamma=\{p_{\Gamma}(U):U\in\Uee_\Gamma\}$ is a base for $p_{\Gamma}(X)$.
\end{itemize}}

We construct by induction a sequence $\{C(m)\}_{m\geq 0}$ of
countable subsets of $A$, and a sequence $\{\mathcal V_m\}_{m\geq
0}$ of countable subfamilies of  $\Sigma_X$ such that:
\begin{itemize}
\item $C_0=B$ and $\mathcal V_0=\{p_B^{-1}(V):V\in\Bee_B\}$, where $\Bee_B$ is a base for $X_B$;
\item $C(m+1)=C(m)\cup\bigcup\{B_U:U\in\mathcal V_m\}$;
\item $\mathcal V_{3m+1}=\mathcal V_{3m}\cup\{\sigma(U_1,..,U_n): U_1,..,U_n\in\mathcal
V_{3m}, n\geq 1\}$;
\item $\mathcal V_{3m+2}=\mathcal
V_{3m+1}\cup\bigcup\{p_C^{-1}(\gamma_C):C\subset C(3m+1){~}\mbox{is
finite}\}$; \item $\mathcal V_{3m+3}=\mathcal
V_{3m+2}\cup\{\bigcap_{i=1}^{i=n}U_i: U_1,..,U_n\in\mathcal
V_{3m+2}, n\geq 1\}$.
\end{itemize}

It is easily seen that the set $\Gamma=\bigcup_{m=0}^{\infty}C_m$
and the family $\Uee_\Gamma=\bigcup_{m=0}^{\infty}\mathcal V_m$
satisfy the conditions (i)-(iii) from Claim 2.

\textit{Claim $3$. The map $p_{\Gamma}\colon X\to X_{\Gamma}$ is a
d-open map.}

It follows from (ii) that
$\Uee_\Gamma=\{p_{\Gamma}^{-1}(V):V\in\Bee_\Gamma\}$. According to
Lemma \ref{2.4}, $\Uee_\Gamma\subset_!\Sigma_X$. Consequently,
$\Uee_\Gamma\subset_!\Tee_X$. Therefore, we can apply Proposition
\ref{d-open} to conclude that $p_{\Gamma}$ is d-open.

Now, consider the family $\Lambda$ of all $\Gamma\in
[A]^{\leq\omega}$ such that there exists a countable family
$\Uee_\Gamma\subset\Sigma_X$ satisfying the conditions (i)-(iii)
from Claim 2. We consider the inverse system $S=\{X_\Gamma,
p^\Gamma_\Theta, \Lambda\}$, where $\Theta\subset\Gamma\in\Lambda$
and $p^\Gamma_\Theta\colon X_\Gamma\to X_\Theta$ is the restriction
of the projection $\pi^\Gamma_\Theta\colon\mathbb I^\Gamma\to\mathbb
I^\Theta$ on the set $X_\Gamma$. Since
$p_{\Theta}=p^\Gamma_\Theta\circ p_{\Gamma}$ and both $p_{\Gamma}$
and $p_{\Theta}$ are d-open surjections, $p^\Gamma_\Theta$ is also
d-open (see Lemma \ref{d-open1}). Moreover, the union of any increasing chain
in $\Lambda$ is again in $\Lambda$. So, $\Lambda$, equipped the
inclusion order, is $\sigma$-complete. Finally, by Claim 2,
$\Lambda$ covers the set $A$. Therefore, the limit of $S$ is a
subset of $\mathbb I^A$ containing $X$ as a dense subset. Hence, $X$
is d-openly generated.

Suppose that $X$ is d-openly generated. So,
$X=\displaystyle\mathrm{a}-\underleftarrow{\lim}S$, where
$S=\{X_\sigma, p^{\sigma}_\varrho, \Gamma\}$ is a $\sigma$-complete
inverse system consisting of separable metric spaces $X_\sigma$ and
d-open surjective bonding maps $p^{\sigma}_\varrho$. Let
$p_\sigma\colon\underleftarrow{\lim}S\to X_\sigma$,
$\sigma\in\Gamma$, be the limit projections and
$q_\sigma=p_\sigma|X$. As in the proof of Theorem \ref{represent}, we can show
that $\Pee=\{\Pee_\sigma:\sigma\in\Gamma\}$ is a club in the family
$\{\Qee\in [\Bee_X]^{\leq\omega}:\Qee\subset_!\Bee_X\}$, where
$\Bee_X=\bigcup\{\Pee_\sigma:\sigma\in\Gamma\}$ and
$\Pee_\sigma=\{q_\sigma^{-1}(V):V\in\Bee_\sigma\}$ with
$\Bee_\sigma$ being a countable base for the topology of $X_\sigma$.
Since $\Bee_X$ consists of co-zero sets, by Corollary \ref{2.8}, the
family $\{\Qee\in [\Sigma_X]^{\leq \omega}: \Qee\subset_!
\Sigma_X\}$ contains also a club. Hence, $X$ is very I-favorable
with respect to the co-zero sets.
\end{proof}

We say that a space $X\subset Y$ is regularly embedded in $Y$ is
there exists a function $\mathrm e\colon\Tee_X\to\Tee_Y$ satisfying
the following conditions for any $U,V\in\Tee_X$:
\begin{itemize}
\item $\mathrm e(\emptyset)=\emptyset$;
\item $\mathrm e(U)\cap X=U$;
\item $\mathrm e(U)\cap\mathrm e(V)=\emptyset$ provided $U\cap
V=\emptyset$.
\end{itemize}

Theorem \ref{4.1} and \cite[Theorem 2.1(ii)]{va1} yield the following
external characterization of very I-favorable spaces with respect to
the co-zero sets (I-favorable spaces with respect to the co-zero
sets have a similar external characterization, see \cite[Theorem
1.1]{vv}).

\begin{cor}\label{4.2}
A completely regular space is very $\mathrm I$-favorable with
respect to the co-zero sets if and only if every $C^*$-embedding of
$X$ in any Tychonoff space $Y$ is regular.
\end{cor}

The next corollary provides an answer of a question from \cite{vv}
whether there exists a characterization of $\kappa$-metrizable
compacta in terms a game between two players.

\begin{cor}\label{4.3}
A compact Hausdorff space is very $\mathrm I$-favorable with respect
to the co-zero sets if and only if $X$ is $\kappa$-metrizable.
\end{cor}

\begin{proof}
A compact Hausdorff space is $\kappa$-metrizable spaces iff $X$ is
the limit space of a $\sigma$-complete inverse system consisting of
compact metric spaces and open surjective bonding maps, see
\cite{s81} and \cite{s76}. Since every d-open surjective map between
compact Hausdorff spaces is open, this corollary follows from
Theorem \ref{4.1}.
\end{proof}

Recall that a normal space is called perfectly normal if every open
set is a co-zero set. So, any perfectly normal spaces is very
I-favorable if and only if it is very I-favorable with respect to
the co-zero sets. Thus, we have the next corollary.

\begin{cor}\label{4.4}
Every perfectly normal very $\mathrm I$-favorable space is d-openly
generated.
\end{cor}

\begin{lem}\label{ro-strategy}
Let $(X,\Tee)$ be a completely regular space. If there is a strong
winning strategy $\sigma' :\bigcup \{ \Tee^n: n\geq 0\} \to \Tee$,
then there is a strong winning strategy $\sigma :\bigcup \{ \Raa^n:
n\geq 0\} \to \Raa $, where $\Raa$ consists of all regular open
subset of $X$.
\end{lem}

\begin{proof}
Assume that $\sigma' :\bigcup \{ \Tee^n: n\geq 0\} \to \Tee $ is a
strong winning strategy. We define a strong winning strategy on
$\Raa$. Let $\sigma(\varnothing)=\int\cl \sigma'(\varnothing)$. We
define by induction $\sigma((V_0, V_1, \ldots, V_k))$,
$V_{k+1}\subset\sigma((V_0, V_1, \ldots, V_k))$, by
$$\sigma((V_0, V_1, \ldots, V_{n+1}))=\int\cl\sigma'((V'_0, V'_1, \ldots, V'_{n+1})),$$
where $V'_{k+1}=V_{k+1}\cap \sigma'((V'_0, V'_1, \ldots, V'_k))$.

Let us show that $\Fee=\{V_n:n\in\omega\}\cup\{\sigma((V_0, V_1,
\ldots, V_{n+1})):n\in \omega\}\subset_!\Raa$. If $\See\subset\Fee$
and $x\not\in\cl\bigcup\See$, let
$$\Fee'=\{V'_n:n\in\omega\}\cup\{\sigma'((V'_0, V'_1, \ldots,
V'_{n+1})):n\in\omega\}$$ and $$\See'=\{W'\in\Fee':W\in\See\}.$$
Note that $\bigcup\See'\subset\bigcup\See$, hence
$x\not\in\cl\bigcup\See'$. So, there is $W'\in\See'$ such that $W'\cap
U'=\emptyset$ for all $U'\in\Fee'$. Assume that
$W'=V_{k+1}\cap\sigma'((V'_0, V'_1, \ldots, V'_{k}))$ and $U'=
V_{i+1}\cap\sigma'((V'_0, V'_1, \ldots, V'_i))$. Then we infer that
$$V_{k+1}\cap\int\cl\sigma'((V'_0, V'_1, \ldots, V'_k))\cap
V_{i+1}\cap\int\cl\sigma'((V'_0, V'_1, \ldots, V'_i))=\emptyset.$$
Since $V_{k+1}\subset\sigma((V_0, V_1, \ldots, V_k))=
\int\cl\sigma'((V'_0, V'_1, \ldots, V'_{k}))$ and $V_{i+1}\subset
\sigma((V_0, V_1, \ldots, V_i))=\int\cl\sigma'((V'_0, V'_1, \ldots,
V'_{i}))$, we get $V_{k+1}\cap V_{i+1}=\emptyset$. Suppose
$W'=V_{k+1}\cap\sigma'((V'_0, V'_1, \ldots, V'_{k}))$ and
$U'=\sigma'((V'_0, V'_1, \ldots, V'_i))$. Then
$$V_{k+1}\cap\int\cl\sigma'((V'_0, V'_1, \ldots, V'_k))\cap
\int\cl\sigma'((V'_0, V'_1, \ldots, V'_i))=\emptyset.$$ So, $W\cap
U=\emptyset$. Similarly, we obtain $W\cap U=\emptyset$ if
$W'=\sigma'((V'_0, V'_1, \ldots, V'_{k}))$ and $U'= \sigma'((V'_0,
V'_1, \ldots, V'_i))$. This completes the proof.
\end{proof}

We say that a topological space $X$ is \textit{perfectly
$\kappa$-normal} if for every open and disjoint   subset $U,V$ there
are open $F_\sigma $ subset $W_U,W_V$ with $W_U\cap W_V=\emptyset$
and $U\subset W_U$ and $V\subset W_V$. It is clear that a space $X$
is perfectly $\kappa$-normal if and only if that each regular open
set in $X$ is $F_\sigma$.

\begin{pro}\label{very}
If a normal perfectly $\kappa$-normal space is a continuous image of
a very $\mathrm I$-favorable space under a perfect map, then $X$ is
d-openly generated.
\end{pro}

\begin{proof}
Every open $F_\sigma$-subset of a normal space is a co-zero set, see
\cite{eng}. So, every regular open subset of a normal and perfectly
$\kappa$-normal space is a co-zero set. Consequently, if $X$ is the
image of very I-favorable space and $X$ is normal and perfectly
$\kappa$-normal, then $X$ is very I-favorable (see Theorem
\ref{k-adic}). Hence, according to Lemma \ref{ro-strategy}, $X$ is a very I-favorable with
respect to the co-zero sets. Finally, Theorem \ref{4.1} implies that $X$
is d-openly generated.
\end{proof}

\begin{cor}\label{4.7}
If the image of a compact Hausdorff very $\mathrm I$-favorable space
under a continuous map is perfectly $\kappa$-normal, then $X$ is
$\kappa$-metrizable.
\end{cor}

Corollary \ref{4.7} implies the following result of Shchepin
\cite[Theorem 18]{s81} which has been proved by different methods:
If the image of a $\kappa$-metrizable compact Hausdorff space $X$
under a continuous map is perfectly $\kappa$-normal, then $X$ is
$\kappa$-metrizable too.

Let us also mention that, according to Shapiro's result \cite{sh},
continuous images of $\kappa$-metrizable compacta have special
spectral representations. This result implies that any such an image
is $\mathrm I$-favorable.



\begin{thebibliography}{40}
\bibitem{byp}
T.~Byczkowski and R.~Pol,  \textit{On the closed graph and open
mapping theorems}, Bull. Acad. Polon. Sci. Sér. Sci. Math. Astronom.
Phys. 24 (1976), no. 9, 723-726.

\bibitem{dkz}
P.~Daniels,~K.~Kunen and H.~Zhou, \textit{On the open-open game},
Fund. Math. \textbf{145} (1994), no. 3, 205--220.

\bibitem{eng}
R.~Engelking, \textit{General topology}, Polish Scientific
Publishers, Warszawa (1977).

\bibitem{kp7}
A.~Kucharski and Sz.~Plewik,  \textit{Game approach to universally
Kuratowski-Ulam spaces}, Topology Appl. \textbf{154} (2007), no. 2,
421--427.

\bibitem{kp8}
A.~Kucharski and Sz.~Plewik,  \textit{Inverse systems and
$I$-favorable spaces},  Topology Appl. \textbf{156} (2008), no. 1,
110--116.

\bibitem{kp9}
A.~Kucharski and Sz.~Plewik,  \textit{Skeletal maps and I-favorable
spaces},   Acta Universitatis Carolinae - Math. et Phys to appear (arXiv:math.GN/10032308).

\bibitem{kun}
K.~Kunen,  \textit{Set theory. An introduction to independence
proofs}, Studies in Logic and the Foundations of Mathematics, 102.
North-Holland Publishing Co., Amsterdam, (1980).

\bibitem{mr}
J.~Mioduszewski and L.~Rudolf,  \textit{$H$-closed and extremally
disconnected Hausdorff spaces},  Dissertationes Math. 66 (1969).

\bibitem{sh} L.~Shapiro, \textit{On a spectral representation of
images of $\kappa$-metrizable bicompacta}, Uspehi Mat. Nauk
\textbf{37} (1982), no. 2(224), 245--246 (in Russian).

\bibitem{s76}
E.~Shchepin, \textit{Topology of limit spaces with uncountable
inverse spectra}, Uspekhi Mat. Nauk \textbf{31} (1976), no. 5(191),
191--226.

\bibitem{s81}
E.~Shchepin, \textit{Functors and uncountable powers of compacta},
Uspekhi Mat. Nauk \textbf{36} (1981), no. 3(219), 3--62 (in
Russian).

\bibitem{tk}
M.~Tkachenko, \textit{Some results on inverse spetra II}, Comment.
Math. Univ. Carol. \textbf{22} (1981), no. 4, 819--841.

\bibitem{va1}
V.~M.~Valov,\textit{Some characterizations of the spaces with a
lattice of $d$-open mappings}, C. R. Acad. Bulgare Sci. 39 (1986),
no. 9, 9-12.


\bibitem{vv}
V.~Valov, \textit{External characterization of I-favorable
spaces},\\
arXiv:math.GN/10050074.

\end{thebibliography}
\end{document}